\documentclass{amsart}
\usepackage{amsmath}
\usepackage{amsthm}
\usepackage{amssymb}
\usepackage{mathrsfs}
\usepackage{hyperref}
\usepackage{chngcntr}
\usepackage{pifont}
\usepackage{tikz}
\usetikzlibrary{matrix,arrows}
\usepackage{stackengine}
\stackMath
\setstackgap{S}{1pt}
\usepackage{wasysym}

\counterwithin{equation}{section}

\makeatletter
\newtheorem*{rep@theorem}{\rep@title}
\newcommand{\newreptheorem}[2]{%
\newenvironment{rep#1}[1]{%
 \def\rep@title{#2 \ref{##1}}%
 \begin{rep@theorem}}%
 {\end{rep@theorem}}}
\makeatother

\newtheorem{thm}{Theorem}[section]
\newreptheorem{thm}{Theorem}
\newtheorem{prop}[thm]{Proposition} 
\newreptheorem{prop}{Proposition}

\newtheorem{lem}[thm]{Lemma}
\newreptheorem{lem}{Lemma}
\newtheorem{cor}[thm]{Corollary}

\theoremstyle{definition}
\newtheorem{dfn}[thm]{Definition}
\newtheorem{exmpl}[thm]{Example}
\newtheorem{?}[thm]{Question}

\theoremstyle{remark}
\newtheorem{rmk}[thm]{Remark}

\newcommand{\tql}{\textquotedblleft}
\newcommand{\tqr}{\textquotedblright}
\newcommand{\noin}{\noindent}
\newcommand{\mc}{\mathcal}
\newcommand{\mb}{\mathbb}
\newcommand{\mbf}{\mathbf}
\newcommand{\openstar}{\text{\ding{73}}}
\newcommand{\peq}{\preceq}

\newcommand{\nci}{\Shortstack{. . . .}} 
\newcommand{\bb}{$\bullet\bullet$}

\begin{document}
   \def\lfhook#1{\setbox0=\hbox{#1}{\ooalign{\hidewidth
    \lower1.5ex\hbox{'}\hidewidth\crcr\unhbox0}}} 
\title{Graph products of completely positive maps}

\author{Scott Atkinson}
\thanks{The author received partial support from NSF Grant \# DMS-1362138.}

\address{Vanderbilt University, Nashville, TN, USA}
\email{scott.a.atkinson@vanderbilt.edu}

\begin{abstract}
We define the graph product of unital completely positive maps on a universal graph product of unital $C^*$-algebras and show that it is unital completely positive itself.  To accomplish this, we introduce the notion of the non-commutative length of a word, and we obtain a Stinespring construction for concatenation. This result yields the following consequences. The graph product of positive-definite functions is positive-definite. A graph product version of von Neumann's Inequality holds. Graph independent contractions on a Hilbert space simultaneously dilate to graph independent unitaries.
\end{abstract}
\maketitle

\section{Introduction}

In operator algebras, graph products unify the notions of free products and tensor products.  In particular, given a simplicial graph $\Gamma = (V,E)$ assign an algebra to each vertex. If there is an edge between two vertices then the two corresponding algebras commute with each other in the graph product; if there is no edge between two vertices then the two corresponding algeras have no relations with each other within the graph product.  Thus free products are given by edgeless graphs, and tensor products are given by complete graphs.

Such products were initially studied in the group theory context where the most prominent examples are the so-called right-angled Artin groups (RAAGs), first introduced by Baudisch in \cite{baudisch}, and right-angled Coxeter groups, first introduced by Chiswell in \cite{chiswell}.  One of the most high-profile appearances of RAAGs is their role in the article \cite{hagwis} by Haglund-Wise whose results are utilized in Agol's celebrated resolution of the virtual Haken conjecture \cite{agol}. There has been extensive work on this subject in group theory, and we cannot possibly acknowledge all of the significant contributions to the topic. A very incomplete list of some notable references in the group context are Droms's series of papers \cite{droms3,droms2,droms1},  Green's general treatment \cite{green}, Januskiewicz's representation theoretic result \cite{janus}, Valette's weak amenability result \cite{valette},  Charney's survey \cite{charney}, and Wise's book \cite{wise}.

Graph products have been recently imported into operator algebras by several authors under just about as many names.  M\l{}otkowski developed some of the theory under the name \tql$\Lambda$-free probability\tqr in the context of non-commutative probability in \cite{mlot}.  In \cite{spewys}, Speicher-Wysocza\'nski revived M\l{}otkowski's work, looking at the related cumulant combinatorics and calling the idea \tql$\varepsilon$-independence.\tqr\;  Independently, in \cite{casfim}, Caspers-Fima drew inspiration directly from Green's thesis \cite{green} and took a foundational approach to graph products from both operator algebraic and quantum group theoretic perspectives.

The purpose of the present paper is to write down a graph product of unital completely positive maps and show that it is again unital completely positive in the spirit of \cite{boca}. This was done particularly for graph products of finite von Neumann algebras in Proposition 2.30 of \cite{casfim} in order to prove that the Haagerup property is preserved under taking graph products.  This article gives the result for the much more general $C^*$-algebraic setting.

The strategy for proving the main result, Theorem \ref{mainthm}, is largely combinatorial.  While there are alternative avenues potentially available (especially in light of the recent preprint \cite{davkak}), the appeal of the approach in this article is the development of some tools addressing the less-familiar combinatorics presented by graph products. In particular, in \S\S\ref{ncl}, we introduce the notion of the \emph{non-commutative length} of a reduced word in a graph product (see Definition \ref{nclength}). Just as the length of a word is an indispensable tool in the theory of free products, the non-commutative length of a word in a graph product can be used analogously to organize arguments by ignoring, in a sense, letters that commute.  In fact, in the free product (edgeless graph) case, the two notions essentially coincide--see Remark \ref{lengths}. Additionally, in \S\S\ref{stine}, we develop a Stinespring construction for concatenation within a finite subset of words in a graph product. This construction yields a version of Schwarz's Inequality for our setting.   Immediately after the proof of Theorem \ref{mainthm}, in \S\S\ref{tpe}, we illustrate how our proof strategy applies in the complete graph case; this gives a new combinatorial proof of the fact that the tensor product of ucp maps on a max tensor product is again ucp.

Following the festival of induction in \S\ref{gpmaps}, we record several consequences in \S\ref{cons}.  The first is Corollary \ref{choda}, giving the graph product analog of Choda's main result from \cite{choda} .  Next, we present Theorem \ref{posdef} which states that the graph product of positive-definite functions on a graph product of groups is itself positive-definite.  We conclude the paper with some results regarding unitary dilation in the graph product context.  In particular, we obtain graph product versions of the Sz.-Nagy-Foia\lfhook{s} dilation theorem (Theorem \ref{dil}), von Neumann's Inequality (Corollary \ref{vnineq}), and unitary dilation of graph independent contractions (Theorem \ref{gpdilation}).

\section{Preliminaries}

Fix a simplicial (i.e. undirected, no single-vertex loops, at most one edge between vertices) graph $\Gamma = (V, E)$, where $V$ denotes the set of vertices of $\Gamma$ and $E \subset V \times V$ denotes the set of edges of $\Gamma$. Given discrete groups $\left\{G_v\right\}_{v \in V}$ one can define the graph product of the $G_v$'s as follows.

\begin{dfn}[\cite{green, casfim}]
The graph product $\bigstar_\Gamma G_v$ is given by the free product $* G_v$ modulo the relations $[g,h] =1$ whenever $g \in G_v, h \in G_w$ and $(v,w) \in E$.
\end{dfn}

In the context of $C^*$-algebras, per usual, there are two flavors of graph products: universal and reduced. Some set-up is in order before presenting these constructions.  Both \cite{mlot} and \cite{casfim} present cosmetically differing constructions of the same objects, but since we are adhering to the language of graphs, we will draw primarily from the discussion in \cite{casfim}.

When working with graph products, the bookkeeping can be done by considering words with letters from the vertex set $V$.  Such words are given by finite sequences of elements from $V$ and will be denoted with bold letters.  In order to encode the commuting relations given by $\Gamma$, we consider the equivalence relation generated by the following relations.
\begin{align*}
(v_1,\dots, v_i, v_{i+1},\dots, v_n) &\sim (v_1,\dots, v_i, v_{i+2}, \dots, v_n) &\text{if} &&v_i = v_{i+1}\\
(v_1,\dots, v_i, v_{i+1},\dots, v_n) &\sim (v_1,\dots, v_{i+1}, v_i,\dots, v_n) &\text{if} &&(i,i+1) \in E.
\end{align*}
The concept of a reduced word is central to the theory of graph products.  The following definition is Definition 3.2 of \cite{spenic} in graph language; the equivalent definition in \cite{casfim} appears differently.

\begin{dfn}\label{redv}
A word $\mbf{v} = (v_1,\dots,v_n)$ is \emph{reduced} if whenever $v_k = v_l, k < l$, then there exists a $p$ with $k< p < l$ such that $(v_k, v_p) \notin E$.  Let $\mc{W}_\text{red}$ denote the set of all reduced words.  We take the convention that the empty word is reduced.
\end{dfn}

\begin{prop}[\cite{green,casfim}]\label{reducedlemma}\hspace*{\fill}
\begin{enumerate}
\item Every word $\mbf{v}$ is equivalent to a reduced word $\mbf{w} = (w_1,\dots, w_n)$.  (We let $|\mbf{w}|=n$ denote the \emph{length} of the reduced word.)
\item If $\mbf{v} \sim \mbf{w}\sim\mbf{w}'$ with both $\mbf{w}$ and $\mbf{w}'$ reduced, then the lengths of $\mbf{w}$ and $\mbf{w}'$ are equal and $\mbf{w}' = (w_{\sigma(1)},\dots,w_{\sigma(n)})$ is a permutation of $\mbf{w}$.  Furthermore, this permutation $\sigma$ is unique if we insist that whenever $w_k = w_l, k< l$ then $\sigma(k) < \sigma(l)$.
\end{enumerate}
\end{prop}

\noin Let $\mc{W}_\text{min}$ be a set of representatives of every reduced word such that each equivalence class has exactly one representative in $\mc{W}_\text{min}$.  An element of $\mc{W}_\text{min}$ is called a \emph{minimal word}.  

\subsection{Universal graph products} To define universal graph products we follow the discussion from \cite{mlot} which gives a more constructive definition compared to the equivalent definition appearing in \cite{casfim}.

\begin{dfn}\label{unidef}
Given a graph $\Gamma = (V,E)$ and unital $C^*$-algebras $\mc{A}_v$ for every $v \in V$, the \emph{universal graph product $C^*$-algebra} is the unique unital $C^*$-algebra $\bigstar_\Gamma \mc{A}_v$ together with unital $*$-homomorphisms $\iota_v: \mc{A}_v \rightarrow \bigstar_\Gamma \mc{A}_v$  satisfying the following universal property.
\begin{enumerate}
\item $\iota_v(a)\iota_w(b) = \iota_w(b)\iota_v(a)$ whenever $a \in \mc{A}_v, b \in \mc{A}_w, (v,w) \in E$;
\item For any unital $C^*$-algebra $\mc{B}$ with $*$-homomorphisms $f_v: \mc{A}_v \rightarrow \mc{B}$ such that $f_v(a)f_w(b) = f_w(b)f_v(a)$ whenever $a \in \mc{A}_v, b \in \mc{A}_w, (v,w) \in E$, there exists a unique $*$-homomorphism $\bigstar_\Gamma f_v: \bigstar_\Gamma \mc{A}_v \rightarrow \mc{B}$ such that $\bigstar_\Gamma f_v \circ \iota_{v_0} = f_{v_0}$ for every $v_0 \in V$.
\end{enumerate}
\end{dfn}

The graph product $\bigstar_\Gamma \mc{A}_v$ is the universal $C^*$-algebraic free product $*_{v \in V} \mc{A}_v$ modulo the ideal generated by the commutation relations encoded in the graph $\Gamma$.

The following constructive description of universal graph product $C^*$-algebras also appears in \cite{mlot}.  Ignoring the norm topology, we can consider the universal $*$-algebraic graph product of the $\mc{A}_v$'s, $\openstar_\Gamma \mc{A}_v$, as the universal $*$-algebraic free product of the $\mc{A}_v$'s modulo the ideal generated by the commutation relations coming from the graph $\Gamma$. For each $v \in V$ fix a state $\varphi_v \in S(\mc{A}_v)$, and let $\mathring{\mc{A}}_v = \ker(\varphi_v)$. For each $\mbf{v} = (v_1,\dots,v_n) \in \mc{W}_\text{min}$ let $\mc{A}_\mbf{v} = \mathring{\mc{A}}_{v_1}\otimes \cdots \otimes \mathring{\mc{A}}_{v_n}$ with $\mc{A}_e = \mb{C}1$ where $e$ is the empty word.  We can identify $\openstar_\Gamma \mc{A}_v$ (as a vector space) with the following direct sum of tensor products. \[\openstar_\Gamma \mc{A}_v = \bigoplus_{\mbf{v} \in \mc{W}_\text{min}} \mc{A}_\mbf{v}\]  Then the $C^*$-algebraic graph product $\bigstar_\Gamma \mc{A}_v$ is the $C^*$-envelope of the $*$-algebraic graph product $\openstar_\Gamma \mc{A}_v$.  Compare this with the discussion in Sections 1.2 and 1.4 of \cite{vodyni}.

\begin{dfn}
A \emph{reduced word} $a \in \openstar_\Gamma \mc{A}_v$ is an element of the form $a = a_1\cdots a_m$ where $a_k \in \mathring{\mc{A}_{v_k}}$ and $(v_1,\dots,v_m) \in \mc{W}_\text{red}$.  In such an instance we write $(v_1,\dots, v_m) = \mbf{v}_a$ and say $|a| = m$--denoting the \emph{length} of $a$ (well-defined by Proposition \ref{reducedlemma}).  Accepting the common risks of abusing notation, we let $\mc{W}_\text{red}$ also denote the set of reduced words in $\openstar_\Gamma \mc{A}_v$.  The linear span of $\mc{W}_\text{red} \cup \left\{1\right\}$ is dense in $\bigstar_\Gamma \mc{A}_v$ (see \cite{mlot}). 
\end{dfn}

\subsection{Reduced graph products} The following construction can be found in \cite{casfim}. The reduced graph product of $C^*$-algebras is defined in the presence of states and depends on the construction of a graph product of Hilbert spaces, defined in a way similar to that of the definition of a free product of Hilbert spaces.  

For each $v \in V$ let $\mc{H}_v$ be a Hilbert space with a distinguished unit vector $\xi_v \in \mc{H}_v$.  Put $\mathring{\mc{H}}_v := \mc{H}_v \ominus \mb{C}\xi_v$.  Given $\mbf{v} = (v_1,\dots,v_n) \in \mc{W}_\text{red}$, define \[\mc{H}_\mbf{v} := \mathring{\mc{H}}_{v_1} \otimes \cdots \otimes \mathring{\mc{H}}_{v_n}.\]  If $\mbf{v}, \mbf{w} \in \mc{W}_\text{red}$ with $\mbf{v}\sim\mbf{w}$ then by Proposition \ref{reducedlemma} there is a uniquely determined unitary $\mc{Q}_{\mbf{v},\mbf{w}}: \mc{H}_v \rightarrow \mc{H}_w$.  Since each reduced word $\mbf{v}$ has a unique representative $\mbf{v}' \in \mc{W}_\text{min}$, we write $\mc{Q}_\mbf{v}$ instead of $\mc{Q}_{\mbf{v},\mbf{v}'}$.  

\begin{dfn}
Define the graph product Hilbert space $(\bigstar_\Gamma\mc{H}_v,\Omega)$ as follows. \[\bigstar_\Gamma\mc{H}_v := \mb{C}\Omega \oplus \bigoplus_{\mbf{w} \in \mc{W}_\text{min}} \mc{H}_\mbf{w}\] 
\end{dfn}

Next, given $v_0 \in V$ we define a canonical (left) representation of $B(\mc{H}_{v_0})$ in $B(\bigstar_\Gamma \mc{H}_v)$.  Let $\mc{W}_l(v_0)\subset \mc{W}_\text{min}$ be the set of minimal words $\mbf{w}$ such that $v_0\mbf{w}$ is still reduced.  Put \[\mc{H}_l(v_0):= \mb{C}\Omega \oplus \bigoplus_{\mbf{w} \in \mc{W}_l(v_0)} \mc{H}_\mbf{w}.\]  We have that $\bigstar_\Gamma \mc{H}_v \cong \mc{H}_{v_0} \otimes \mc{H}_l(v_0)$ via the unitary $U_l(v_0)$ defined as follows.
\begin{align*}
U_l(v_0): \mc{H}_{v_0} \otimes \mc{H}_l(v_0) &\rightarrow \bigstar_\Gamma \mc{H}_v\\
 \xi_{v_0} \otimes \Omega &\mapsto \Omega\\
 \mathring{\mc{H}}_{v_0} \otimes \Omega &\mapsto \mathring{\mc{H}}_{v_0}\\
 \xi_{v_0} \otimes \mc{H}_\mbf{w} &\mapsto \mc{H}_\mbf{w}\\
\mathring{\mc{H}}_{v_0}\otimes \mc{H}_\mbf{w} &\mapsto \mc{Q}_{v_0\mbf{w}}(\mathring{\mc{H}}_{v_0} \otimes \mc{H}_\mbf{w})
\end{align*}
\noin Then we define $\lambda_{v_0}: B(\mc{H}_{v_0}) \rightarrow B(\bigstar_\Gamma \mc{H}_v)$ by \[\lambda_{v_0}(x) = U_l(v_0)(x\otimes 1)U_l(v_0)^*.\]

\begin{dfn}
For each $v \in V$ let $\mc{A}_v$ be a unital $C^*$-algebra, let $\varphi_v \in S(\mc{A}_v)$ be a state, and let $(\pi_v, \mc{H}_v, \xi_v)$ be the corresponding GNS triple.  The \emph{(left) reduced graph product $C^*$-algebra} is denoted $\bigstar_\Gamma (\mc{A}_v, \varphi_v)$ and is defined to be the $C^*$-subalgebra in $B(\bigstar_\Gamma \mc{H}_v)$ generated by $\left\{\lambda_v(\pi_v(\mc{A}_v))\right\}_{v \in V}$.  The vector state $\left\langle \cdot \Omega| \Omega\right\rangle$ on $\bigstar_\Gamma (\mc{A}_v, \varphi_v)$ is the reduced graph product state denoted $\bigstar_\Gamma \varphi_v$.  
\end{dfn}

\begin{rmk}\label{B}
As outlined in \cite{casfim}, one can analogously construct right representations $\rho_{v_0}: B(\mc{H}_{v_0}) \rightarrow B(\bigstar_\Gamma \mc{H}_v)$ and subsequently define a right reduced graph product $C^*$-algebra.  
\end{rmk}


\subsection{Graph independence}

We briefly discuss graph products in the context of non-commutative probability.  Compare this discussion with \cite{mlot,spewys}.

\begin{dfn}
A \emph{non-commutative probability space} is given by a pair $(\mc{A}, \varphi)$ where $\mc{A}$ is a unital $C^*$-algebra and $\varphi \in S(\mc{A})$ is a state on $\mc{A}$.
\end{dfn}

\begin{dfn}
Given a non-commutative probability space $(\mc{A},\varphi)$ and a graph $\Gamma = (V,E)$, let $\left\{\mc{A}_v \right\}_{v\in V} \subset \mc{A}$ be a family of unital $C^*$-subalgebras.  Put $\mathring{\mc{A}}_v:= \ker(\varphi|_{\mc{A}_v})$.  An element $a \in C^*(\cup_{v\in V} \mc{A}_v)$ is \emph{reduced with respect to $\varphi$} if $a = a_1\cdots a_m$ where $a_j \in \mathring{\mc{A}}_{v_j}$ for $1 \leq j \leq m$ and $(v_1,\dots,v_m)$ is reduced in the sense of Definition \ref{redv}.
\end{dfn}

\begin{dfn}\cite{mlot,spewys}
Given a non-commutative probability space $(\mc{A},\varphi)$ and a graph $\Gamma = (V,E)$, a family of unital $C^*$-subalgebras $\left\{\mc{A}_v\right\}_{v\in V}\subset (\mc{A},\varphi)$ is \emph{$\Gamma$ independent} (or \emph{graph independent} when context is clear) if
\begin{enumerate}
\item $(v,v') \in E \Rightarrow \mc{A}_v$ and $\mc{A}_{v'}$ commute;
\item for any $a \in C^*(\cup_{v\in V} \mc{A}_v)$ such that $a$ is reduced with respect to $\varphi, \varphi(a) = 0$.
\end{enumerate}
A family of random variables $\left\{x_v\right\}_{v\in V} \subset \mc{A}$ is $\Gamma$ independent if the family of their generated unital $C^*$-algebras $\left\{C^*(1,x_v)\right\}_{v \in V}$ is $\Gamma$ independent.
\end{dfn}

\begin{exmpl}
By construction, $\left\{\lambda_v(\pi_v(\mc{A}_v))\right\}_{v\in V} \subset (\bigstar_\Gamma (\mc{A}_v,\varphi_v), \bigstar_\Gamma \varphi_v)$ is $\Gamma$ independent.
\end{exmpl}

Consider the following analog of Lemma 5.13 of \cite{spenic}.

\begin{lem}\label{stres}
Let $(\mc{A},\varphi)$ be a non-commutative probability space.  Let $\Gamma = (V,E)$ be a graph, and let the unital subalgebras $\mc{A}_v, v \in V$, be $\Gamma$ independent in $(\mc{A}, \varphi)$.  Let $\mc{B}$ be the $C^*$-algebra generated by the $A_v$'s.  Then $\varphi|_\mc{B}$ is uniquely determined by $\varphi|_{\mc{A}_v}$ for all $v \in V$.
\end{lem}

\begin{proof}
This follows directly from (the proof of) Lemma 1 in \cite{mlot}.
\end{proof}

\begin{rmk}
Although this is not the topic of the present paper,  we note that the existence of left and right (cf. Remark \ref{B}) representations on graph product Hilbert spaces sets the stage for an investigation into \tql bi-graph independence\tqr--see \cite{pof1}.
\end{rmk}

\section{Graph products of maps}\label{gpmaps}

This section presents the main result of the present article, establishing the existence of graph products of unital completely positive maps. The max tensor product and the universal free product are both examples of universal graph products; so the following result is a generalization and unification of the max tensor product and Boca's universal free product of completely positive maps appearing in \cite{boca}.


Let $\Gamma = (V, E)$ be a graph.  Let $\mc{B}$ be a unital $C^*$-algebra. For each $v \in V$,  let $\mc{A}_v$ be a unital $C^*$-algebra, and let $\theta_v: \mc{A}_v \rightarrow \mc{B}$ be a unital completely positive map with the property that if $(v,v') \in E$ then $\theta_v(\mc{A}_v)$ commutes with $\theta_{v'}(\mc{A}_{v'})$. Furthermore, for each $v_0 \in V$, fix a state $\varphi_{v_0} \in S(\mc{A}_{v_0})$, and let $\iota_{v_0}: \mc{A}_{v_0} \rightarrow \bigstar_\Gamma \mc{A}_v$ be the inclusion given in Defintion \ref{unidef}. We densely define the unital graph product map $\bigstar_\Gamma \theta_v$ with respect to the states $\varphi_v$ on $\mc{W}_\text{red} \cup \left\{1\right\}$ and extend linearly.  For $a_j \in \mathring{\mc{A}}_{v_j}:=\ker(\varphi_{v_j}), 1 \leq j \leq n, (v_1,\dots, v_n)  \in \mc{W}_\text{red}$, 
\begin{align}
\bigstar_\Gamma \theta_v \Big(\prod_{j=1}^n \iota_{v_j}(a_j)\Big) := \prod_{j=1}^n \iota_{v_j}(\theta_{v_j}(\iota_{v_j}(a_j))).\label{ucpdef}
\end{align}
From now on, we suppress the $\iota_v$'s.

\begin{thm}\label{mainthm}
The map $\bigstar_\Gamma \theta_v$ densely defined on the linear span of $\mc{W}_\text{red}\cup\left\{1\right\}$ by the relation \eqref{ucpdef} extends by continuity to a unital completely positive map $\bigstar_\Gamma \mc{A}_v \rightarrow \mc{B}$.
\end{thm}

The proof we present at the end of this section is an adaptation of Boca's original proof in \cite{boca}. It deserves mentioning that a recently posted preprint (\cite{davkak}) by Davidson-Kakariadis exhibits an alternative proof of the corresponding result in the amalgamated free product case using a dilation theoretic approach.  While a graph product companion to Davidson-Kakariadis's technique is worth pursuing, generalizing Boca's strategy to the graph product setting has the benefit of developing some tools and facts regarding the less familiar--and sometimes frustrating--combinatorics of graph products. Due to the subtlety of the combinatorics, some preparation is in order.  

For the sake of simpler notation we will denote $\bigstar_\Gamma \theta_v$ by $\Theta$. As in \cite{boca} assume that $\mc{B} \subset B(\mc{H})$ for some Hilbert space $\mc{H}$ and that $I_\mc{H} \in \mc{B}$.  It is well-known that it suffices to show that for any $n \in \mb{N}, x_1, \dots, x_n \in \bigstar_\Gamma \mc{A}_v, \xi_1, \dots, \xi_n \in \mc{H}$, \[\sum_{i,j =1}^n \langle \Theta(x_i^*x_j)\xi_j | \xi_i \rangle \geq 0.\]  By an argument identical to the one in \cite{boca}, we can further reduce the required inequality to the following.  It is enough to check that for any finite set $X$ in $\mc{W}_\text{red} \cup \left\{1\right\}$ and any function $\xi: X \rightarrow \mc{H}$, we have \[\sum_{x,y \in X} \langle \Theta(x^*y)\xi(y) | \xi(x)\rangle \geq 0.\]

Although the following fact is very simple, it deserves to be recorded separately because it is so fundamental to the proceeding arguments.

\begin{prop}
Let $x, y \in \mc{W}_\text{red}$.  If $x^*y$ is not reduced, there exist orderings of $\mbf{v}_x = (v_1,\dots, v_n)$ and $\mbf{v}_y= (v'_1,\dots, v'_m)$ such that $v_1 = v'_1$.
\end{prop}

\subsection{Non-commutative length}\label{ncl}

We now discuss useful tools for the relevant combinatorics of this question.

\begin{dfn}\label{complete}
A finite subset $X \subset \mc{W}_\text{red} \cup \left\{ 1 \right\}$ is \emph{complete} if $1 \in X$ and whenever $a_1\cdots a_m \in X$ we have $a_{\sigma(2)}\cdots a_{\sigma(m)} \in X$ and  $a_{\sigma(1)} \cdots a_{\sigma(m-1)} \in X$ for every permutation $\sigma \in S_m$ such that $a_1 \cdots a_m = a_{\sigma(1)} \cdots a_{\sigma(m)}$.  In other words $X$ is complete if it contains the unit and is closed under left and right truncations of any equivalent rearrangements.  Compare this to Boca's definition of a complete set in \cite{boca}.  Let $\mbf{v}_X:= \left\{\mbf{v} \in \mc{W}_\text{red} | \mbf{v} = \mbf{v}_a \text{ for some } a \in X\right\}$.
\end{dfn}

Since every finite set in $\mc{W}_\text{red} \cup \left\{ 1\right\}$ is contained in a complete set, we can make one final reduction of the desired inequality as follows. For any complete set $X \subset \mc{W}_\text{red} \cup \left\{ 1 \right\}$ and any function $\xi: X \rightarrow \mc{H}$, we have 
\begin{align}
\sum_{x,y \in X} \langle \Theta(x^*y)\xi(y) | \xi(x)\rangle &\geq 0. \label{goal}
\end{align}

\begin{dfn}
We can place a partial order $\peq$ on $\mc{W}_\text{red} \cup \left\{ 1\right\}$ with respect to truncation as follows.  For every $x \in \mc{W}_\text{red}, 1 \peq x$; and given $x, y \in \mc{W}_\text{red}, y \peq x$ if either $x = y$ or $x$ truncates (as in Definition \ref{complete}) to $y$.  This order also applies to the words in $V$. 
\end{dfn}

Let $Y \subset \mc{W}_\text{red}\cup\left\{1\right\}$ be any finite nonempty subset.  Put \[Y^\peq := \left\{x \in \mc{W}_\text{red}\cup\left\{1\right\} | \exists y \in Y: x \peq y \right\}.\]  Clearly, $Y^\peq$ is complete.

\begin{dfn}\label{nclength}
Fix $v_0 \in V$.  Let $\mbf{v} = (v_1,\dots,v_n,v_0)$ be reduced.  We let $\nci \mbf{v} \nci_{v_0}$ denote the \emph{(right-hand) non-commutative length of $\mbf{v}$ with respect to $v_0$}, given by \[\nci \mbf{v} \nci_{v_0} := \text{Card}\Big(\left\{i | 1\leq i \leq n, (v_i,v_0) \notin E\right\}\Big).\]  Note that this counts when $v_0$ is repeated. If $\mbf{v}$ cannot be written with $v_0$ at the right-hand end, put $\nci \mbf{v} \nci_{v_0} = -1$.  If $w \in \bigstar_\Gamma \mc{A}_v$ is reduced, let $\nci w \nci_{v_0} = \nci \mbf{v}_w\nci_{v_0}$.  Given a finite set $X$ of reduced words (of vertices or algebra elements), we define the \emph{(right-hand) non-commutative length of $X$ with respect to $v_0$}, denoted $\nci X \nci_{v_0}$ to be given by \[ \nci X \nci_{v_0} := \max_{w \in X} \nci w \nci_{v_0}.\]
\end{dfn}

\begin{rmk}\label{lengths}
Observe that in a free product (graph product over a graph with no edges), the length of a reduced word is always one more than the non-commutative length of a reduced word.
\end{rmk}

\begin{dfn}\label{stdform}
Fix $v_0 \in V$. Let $\mbf{x} \in \mc{W}_\text{red}$ be such that $v_0 \in \mbf{x}$. Suppose $\mbf{y},\mbf{c},\mbf{b}  \in \mc{W}_\text{red}$, satisfy the following properties. 
\begin{itemize}
	\item $\mbf{x} = \mbf{y}\mbf{c}(v_0)\mbf{b}$;
	\item $\mbf{b}$ is the word of smallest length so that $\mbf{y}\mbf{c}(v_0) \peq \mbf{x}$ and $\nci \mbf{y}\mbf{c}(v_0)\nci_{v_0} = \nci \left\{\mbf{x}\right\}^\peq \nci_{v_0}$;
	\item $\mbf{y}$ is the word of smallest length so that $\mbf{y}(v_0) \peq \mbf{x}$ and $\nci \mbf{y}(v_0) \nci_{v_0} = \nci \left\{\mbf{x}\right\}^\peq \nci_{v_0}$.
\end{itemize}
Then we say that $\mbf{x} = \mbf{y}\mbf{c}(v_0)\mbf{b}$ is in \emph{standard form with respect to $v_0$}.
We extend this definition to reduced words of algebra elements.
\end{dfn}

The following proposition follows from a straightforward induction argument using the fact that truncation preserves standard form; the proof is left as an exercise.

\begin{prop}
If $\mbf{x} = \mbf{y}\mbf{c}(v_0)\mbf{b}$ is in standard form with respect to $v_0$, then the words $\mbf{y}, \mbf{c},$ and $\mbf{b}$ are unique.
\end{prop}

\noin Given $a \in \mc{A}_v$, let $\mathring{a}:= a - \varphi_v(a)1$.  We have the following lemma.

\begin{lem}\label{X1crossterms}
Fix $v_0 \in V$. Let $\mbf{x} \in \mc{W}_\text{red}$ be such that $v_0 \in \mbf{x}$.  Say $\mbf{x} = \mbf{y}\mbf{c}(v_0)\mbf{b}$ is in standard form with respect to $v_0$.  Let $y, c, a, b \in \mc{W}_\text{red} \cup \left\{1\right\}$ be such that $\mbf{v}_y = \mbf{y}, \mbf{v}_c = \mbf{c}, \mbf{v}_b = \mbf{b},$ and $a \in \mathring{\mc{A}}_{v_0}$.  If $\mbf{x}'$ is such that $\nci \left\{\mbf{x}'\right\}^\peq \nci_{v_0} < \nci \left\{\mbf{x}\right\}^\peq \nci_{v_0}$, then for every $x' \in \mc{W}_\text{red}$ such that $\mbf{x}' = \mbf{v}_{x'}$, \[\Theta(b^*a^*c^*y^*x') = \Theta(b^*a^*)\Theta(c^*y^*x').\]

\end{lem}

\begin{proof}
We proceed by induction on $\nci \left\{\mbf{x}\right\}^\peq \nci_{v_0}$. 

\begin{itemize}
	\item $\nci \left\{\mbf{x}\right\}^\peq \nci_{v_0} = 0$: We proceed by further induction on $|\mbf{x}'|$.  
		\begin{itemize}
			\item[\bb] $|\mbf{x}'| = 0$:  $x' = 1$, and the statement is obviously true.  
			
			\item [\bb] $|\mbf{x}'| = k >0$: if $b^*a^*c^*y^*x'$ is reduced then the equality holds.  Suppose $b^*a^*c^*y^*x'$ is not reduced.  In this case $y =1$. Let $c = c_1\cdots c_m$ and $x' = x'_1\cdots x'_k$.  By the definition of standard form, we have that we can rearrange the $c_i$'s and $x'_i$'s so that $\mbf{v}_{c_1} = \mbf{v}_{x'_1}$.  That is, none of the $b$ terms can cross past $a$; otherwise the minimality of $|b|$ would be contradicted.  So we have
\begin{align}
&\Theta(b^* a^* c_m^* \cdots c_1^* x'_1 \cdots x'_k)\notag\\
 & = \Theta(b^*a^*c_m^*\cdots c_2^*(\mathring{c_1^*x'_1})x'_2\cdots x'_k)\notag\\
 & + \varphi_{\mbf{v}_{x'_1}}(c_1^*x'_1)\Theta(b^*a^*c_m^*\cdots c_2^*x'_2\cdots x'_k) \notag \\
 &=  \Theta(b^*a^*)\Theta(c_m^*\cdots c_2^*(\mathring{c_1^*x'_1})x'_2\cdots x'_k)  \label{1}\\
 &+ \varphi_{\mbf{v}_{x'_1}}(c_1^*x'_1)\Theta(b^*a^*)\Theta(c_m^*\cdots c_2^*x'_2\cdots x'_k)\notag \\
 &= \Theta(b^*a^*)\Theta(c_m^* \cdots c_1^* x'_1 \cdots x'_k) \notag 
\end{align}
where \eqref{1} follows from the fact that $\nci \left\{x_2' \cdots x_k\right\}^\peq\nci_{v_0}$ is less than $\nci \left\{(\mathring{c_1^*x_1'})^* c_2 \cdots c_m a b \right\}^\peq\nci_{v_0}$ and $\nci \left\{c_2\cdots c_m a b\right\}^\peq\nci_{v_0}$, and thus the inductive hypothesis gives the desired equality.
		\end{itemize}
		
	\item $\nci \left\{\mbf{x}\right\}^\peq \nci_{v_0} >0$:  Again we induct further on $|\mbf{x}'|$.  
		\begin{itemize}
			\item[\bb] $|\mbf{x}'| =0$: Trivial.  
			
			\item[\bb] $|\mbf{x}'| = k>0$:  If $b^*a^*c^*y^*x'$ is reduced then the equality holds.  Suppose $b^*a^*c^*y^*$ is not reduced, and let $cy = z_1 \cdots z_m$ and $x' = x'_1 \cdots x'_k$.  As before, we can rearrange the $z_i$'s and $x'_i$'s so that $\mbf{v}_{z_1} = \mbf{v}_{x'_1}$.  If $(\mbf{v}_{z_1}, v_0) \in E$  then the argument in the $\nci \left\{\mbf{x}\right\}^\peq \nci_{v_0} = 0$ case holds.  Assume that $(\mbf{v}_{z_1}, v_0) \notin E$.  Then $\nci z_2 \cdots z_m a\nci_{v_0} = \nci yca\nci_{v_0} - 1 \geq 0$.  It is a quick check to see that if $\nci \left\{\mbf{x}'\right\}^\peq\nci_{v_0} \neq -1$ then deleting $x'_1$ from the left decreases the non-commutative length by one, and if $\nci \left\{\mbf{x}'\right\}^\peq\nci_{v_0} = -1$, then deleting $x'_1$ leaves the non-commutative length alone.  In either case, the inductive hypothesis applies, yielding the equality as illustrated above.\qedhere
		\end{itemize}
	\end{itemize}
	
\end{proof}

\begin{lem}\label{Y1crossterms}
Fix $v_0 \in V$. Let $\mbf{x}, \mbf{x}' \in \mc{W}_\text{red}$ be such that $\nci \left\{ \mbf{x}\right\}^\peq \nci_{v_0} = \nci \left\{ \mbf{x}'\right\}^\peq \nci_{v_0} >0$.  Let $\mbf{y}, \mbf{y}', \mbf{c}, \mbf{c}', \mbf{b}, \mbf{b}' \in \mc{W}_\text{red}$ be such that $\mbf{x} = \mbf{y} \mbf{c} (v_0) \mbf{b}$ and $\mbf{x}' = \mbf{y}' \mbf{c}' (v_0) \mbf{b}'$ are both in standard form with respect to $v_0$.  If $\mbf{y} \neq \mbf{y}'$ then for every $y, y', c, c', a, a', b, b' \in \mc{W}_\text{red}\cup \left\{1\right\}$ such that $\mbf{v}_y = \mbf{y}, \mbf{v}_{y'} = \mbf{y}', \mbf{v}_c = \mbf{c}, \mbf{v}_{c'} = \mbf{c}', \mbf{v}_b = \mbf{b}, \mbf{v}_{b'} = \mbf{b}',$ and $a, a' \in \mathring{\mc{A}}_{v_0}$ we have \[\Theta(b^*a^*c^*y^*y'c'a'b') = \Theta(b^*a^*)\Theta(c^*y^*y'c'a'b') (= \Theta(b^*a^*)\Theta(c^*y^*y'c')\Theta(a'b') ) .\]
\end{lem}

\begin{proof}
Let $yc = z_1\cdots z_m$ and $y'c' = z'_1\cdots z'_{m'}$.  We proceed by induction on $\nci \left\{\mbf{x}\right\}^\peq\nci_{v_0}$.
\begin{itemize}
	\item $\nci \left\{\mbf{x}\right\}^\peq \nci_{v_0} = 1$: We induct further on $m + m'$.
	
	\begin{itemize}
		\item[$\bullet \bullet$] $m+m' = 2$: Since $\mbf{y} \neq \mbf{y}'$, we immediately get that $b^*a^*z_1^*z_1'a'b'$ is reduced.  So the equality follows.
		
		\item[$\bullet \bullet$] $m + m' > 2$: If $b^*a^*c^*y^*y'c'a'b'$ is reduced then we are done.  Suppose $b^*a^*c^*y^*y'c'a'b'$ is not reduced.  Then we can rearrange the $z$ and $z'$ terms so that $\mbf{v}_{z_1} = \mbf{v}_{z'_1}$.  Then we have
		\begin{align}
		&\Theta(b^* a^* z_m^* \cdots z_1^* z'_1 \cdots z'_{m'} a' b')\notag\\
		 & = \Theta(b^*a^*z_m^*\cdots z_2^*(\mathring{z_1^*z'_1})z'_2\cdots z'_{m'} a' b') \label{2} \\
 &+ \varphi_{\mbf{v}_{z_1}}(z_1^*z'_1)\Theta(b^*a^*z_m^*\cdots z_2^*z'_2\cdots z'_{m'} a' b') \notag
		\end{align}
		Since $\mbf{y} \neq \mbf{y}'$ we have that $(\mbf{v}_{z_1}, v_0) \in E$.   The inductive hypothesis on $m + m'$ applies, yielding the desired equality.
	\end{itemize}
	
	\item $\nci \left\{\mbf{x}\right\}^\peq \nci_{v_0} >1$: Again, induct further on $m + m'$.
	
		\begin{itemize}
		
			\item[$\bullet \bullet$] $m + m' = 2\nci \left\{\mbf{x}\right\}^\peq \nci_{v_0}$: Suppose $b^*a^*c^*y^*y'c'a'b'$ is not reduced and that $\mbf{v}_{z_1} = \mbf{v}_{z'_1}$.  Then we obtain the same decomposition as in \eqref{2}.  Then by applying Lemma \ref{X1crossterms} to the first term on the right-hand side of \eqref{2} and the inductive hypothesis on $\nci \left\{\mbf{x}\right\}^\peq\nci_{v_0}$ to the second term, we obtain the desired equality.
			
			\item[$\bullet \bullet$] $m + m' > 2\nci\left\{\mbf{x}\right\}^\peq\nci_{v_0}$: Suppose $b^*a^*c^*y^*y'c'a'b'$ is not reduced and that $\mbf{v}_{z_1} = \mbf{v}_{z'_1}$; consider the decomposition from \eqref{2}.  If $(\mbf{v}_{z_1}, v_0) \notin E$, then as in the $m + m' = 2\nci \left\{\mbf{x}\right\}^\peq \nci_{v_0}$ case, apply Lemma \ref{X1crossterms} to the first term on the right-hand side of \eqref{2} and apply the inductive hypothesis on $\nci \left\{\mbf{x}\right\}^\peq\nci_{v_0}$ to the second term.  If $(\mbf{v}_{z_1}, v_0) \in E$, apply the inductive hypothesis on $m + m'$ to both terms on the right-hand side of \eqref{2}.\qedhere		
		\end{itemize}	
\end{itemize}
\end{proof}

\subsection{A Stinespring construction for concatenation}\label{stine}

The goal of this subsection is to show the following generalization of Schwarz's Inequality.

\begin{prop}\label{schwarz}
Let $X \subset \mc{W}_\text{red}\cup \left\{1\right\}$ be a complete set, and assume that for every function $\xi: X \rightarrow \mc{H}$, \eqref{goal} holds.  For $1 \leq i \leq N$, let $b_i, c_i, b_ic_i \in X$.  Then we have the following matrix inequality. \[ \left[ \Theta(b_i^*c_i^*c_jb_j)\right]_{ij} \geq \left[\Theta(b_i^*)\Theta(c_i^*c_j)\Theta(b_j)\right]_{ij}\]\end{prop}

\noin We will prove Proposition \ref{schwarz} by making use of a Stinespring construction for (left-hand) concatenation. Consider $\mb{C}^{|X|}$ with standard basis $\left\{e_x\right\}_{x\in X}$.  The inequality \eqref{goal} implies that we can define a positive semi-definite sesquilinear form on $\mc{H} \otimes \mb{C}^{|X|}$ given by \[ \langle \xi \otimes e_y | \eta \otimes e_x\rangle = \langle \Theta(x^*y) \xi | \eta\rangle.\]  By standard arguments this yields a Hilbert space  that we will denote by $\mc{H} \otimes_\Theta \mb{C}^{|X|}$.  For each $x \in X$ let $V_x: \mc{H} \rightarrow \mc{H}\otimes_\Theta \mb{C}^{|X|}$ be given by $V_x(\xi) = \xi \otimes_\Theta e_x$.  Observe that $V_1$ is an isometry:
\begin{align*}
||V_1 \xi||^2_{\mc{H}\otimes_\Theta \mb{C}^{|X|}} &= \langle \xi \otimes_\Theta e_1 | \xi \otimes_\Theta e_1\rangle \\
& = \langle \Theta(1) \xi | \xi\rangle\\
&= ||\xi||^2_\mc{H}.
\end{align*}

Given $x \in X$ with $|x| = 1$, we define the \emph{left-concatenation operator} $L_x: \mc{H}\otimes_\Theta \mb{C}^{|X|} \rightarrow \mc{H} \otimes_\Theta \mb{C}^{|X|}$ as follows. \[L_x(\xi \otimes_\Theta e_y) = \left\{\begin{array}{lcr}
0 & \text{if} & xy \notin X \\
&&\\
\xi \otimes_\Theta e_{xy} & \text{if} & xy \in X
\end{array}\right.\]  

\begin{prop}\label{Lxbdd}
Let $X \subset \mc{W}_\text{red}\cup \left\{1\right\}$ be a complete set, and assume that for every function $\xi: X \rightarrow \mc{H}$, \eqref{goal} holds. Given $x \in X$ with $|x|=1$, the left-concatenation operator $L_x$ is bounded.
\end{prop}

\noin Proposition \ref{Lxbdd} is all we need to prove Proposition \ref{schwarz}:

\begin{proof}[Proof of Proposition \ref{schwarz}]
Given $a = a_1 \cdots a_m \in X$, Proposition \ref{Lxbdd} provides that the corresponding left-concatenation operator $L_a := L_{a_1}\cdots L_{a_m}$ is bounded.  Evidently, given $x,y \in X,$ \[\Theta(x^*y) = V_1^*L_x^*L_yV_1.\]  Thus we have
\begin{align*}
&\left[ \Theta(b_i^*c_i^*c_jb_j) \right]_{ij} \\
& = \left[V_1^*L_{b_i}^*L_{c_i}^*L_{c_j}L_{b_j}V_1 \right]_{ij}\\
&=\text{Dg}(V_1^*L_{b_i}^*)  \left[
	\begin{matrix}
	L_{c_1}^*\\
	\vdots  \\
	L_{c_N}^* 
	\end{matrix}\right]
	\left[\begin{matrix}
	L_{c_1} & \cdots & L_{c_N}
	\end{matrix}\right]
	\text{Dg}(L_{b_j}V_1)\\
	&\geq \text{Dg}(V_1^*L_{b_i}^*)  \left(\text{Dg}(V_1V_1^*)\left[
	\begin{matrix}
	L_{c_1}^*\\
	\vdots  \\
	L_{c_N}^* 
	\end{matrix}\right]
	\left[\begin{matrix}
	L_{c_1} & \cdots & L_{c_N}
	\end{matrix}\right]
	\text{Dg}(V_1V_1^*)\right) \text{Dg}(L_{b_j}V_1)\\
&= \left[V_1^*L_{b_i}^*V_1V_1^*L_{c_i}^*L_{c_j}V_1V_1^*L_{b_j}V_1\right]_{ij}\\
&= \left[\Theta(b_i^*)\Theta(c_i^*c_j)\Theta(b_j)\right]_{ij}
\end{align*}
where $\text{Dg}(\bullet)$ is the $N \times N$ matrix with the diagonal given by $\bullet$.
\end{proof}

\noin Thus we have reduced the goal of the current subsection to proving Proposition \ref{Lxbdd}.  We accomplish this by making one last reduction. The following technical lemma can be used to prove Proposition \ref{Lxbdd}. 

\begin{lem}\label{techlem}
Let $X \subset \mc{W}_\text{red}\cup \left\{1\right\}$ be a complete set with $|X|\geq 2$, and assume that for every function $\xi: X \rightarrow \mc{H}$, \eqref{goal} holds.  Let $(v_0) \in \mbf{v}_X$ and let $y \in X$ be such that $(v_0)\cdot\mbf{v}_y \in \mbf{v}_X$.  For \emph{any} $a \in \mc{A}_{v_0}$, \[\Theta(y^*a^*ay) \geq \Theta(y^*)\Theta(a^*a)\Theta(y) = \Theta(y^*)\theta_{v_0}(a^*a)\Theta(y) \geq 0.\]
\end{lem}

\begin{proof}[Proof of Proposition \ref{Lxbdd}]
Let $x \in X$ be such that $|x| = 1$, and let $y \in X$ be such that $xy \in X$.  We have that $||x^*x|| - x^*x \geq 0$, so there is some $a \in \mc{A}_{\mbf{v}_x}$ such that $a^*a = ||x^*x||-x^*x$.  Then by Lemma \ref{techlem}, 
\[
\Theta(y^*(||x^*x||-x^*x)y) = \Theta (y^*a^*ay) \geq \Theta(y^*)\Theta(a^*a)\Theta(y) \geq 0.\]
Thus
\begin{align*}
||L_x(\xi \otimes_\Theta e_y)||^2_{\mc{H}\otimes_\Theta \mb{C}^{|X|}} &= ||\xi \otimes_\Theta e_{xy}||^2_{\mc{H}\otimes_\Theta \mb{C}^{|X|}}\\
& = \langle \Theta(y^*x^*xy)\xi | \xi\rangle \\
& \leq \langle \Theta(y^*||x^*x||y) \xi | \xi \rangle\\
& = ||x||^2 ||\xi \otimes_\Theta e_y||^2_{\mc{H}\otimes_\Theta \mb{C}^{|X|}}.\qedhere
\end{align*}
\end{proof}

\begin{proof}[Proof of Lemma \ref{techlem}]
We proceed by induction on $|X|$.
\begin{itemize}
	\item $|X| =2$: Then $|X| = \left\{1, a\right\}$, and for $y$ to satisfy the hypothesis, $y =1$.  So the statement holds trivially.
	
	\item $|X| > 2$: We induct further on $|y|$.
		\begin{itemize}
			\item[\bb] $|y| = 0$: Trivial.
			
			\item[\bb] $|y| >0$: Let $y = y_1\cdots y_m$ so that $y_j \in \mathring{\mc{A}}_{v_j}, 1 \leq j \leq m$.  If for every $1 \leq j \leq m, (v_0,v_j) \in E$, then \[\Theta(y^*a^*ay) = \theta_{v_0}(a^*a)\Theta(y^*y).\]  Consider the complete set $X' := \left\{y\right\}^\peq$.  Since $\left\{1\right\} \subsetneq X' \subsetneq X$, we have that $X'$ is a complete set with $|X'|\geq 2$ such that for every function $\xi: X \rightarrow \mc{H}$, \eqref{goal} holds.  By the inductive hypothesis on the cardinality of the complete set and the proofs of Propositions \ref{Lxbdd} and \ref{schwarz}, we have that \[\Theta(y^*y) \geq \Theta(y^*)\Theta(y).\]  Because $\theta_{v_0}(a^*a)$ is positive and $\theta_{v_0}(a^*a)$ and $\Theta(y^*y) - \Theta(y^*)\Theta(y)$ commute, we have that 
			\begin{align*}
			\Theta(y^*a^*ay) &= \theta_{v_0}(a^*a)\Theta(y^*y)\\
			&\geq \theta_{v_0}(a^*a)\Theta(y^*)\Theta(y) \\
			&= \Theta(y^*)\theta_{v_0}(a^*a)\Theta(y).
			\end{align*}
			If there exists $1 \leq j \leq m$ such that $(v_0,v_j) \notin E$, let $1 \leq J \leq m$ be the largest index (among all equivalent permutations) such that $v_J = v_j$.  Consider
			\begin{align*}
			\hspace{.8cm}\Theta(y_m^* &\cdots y_1^*a^*ay_1 \cdots  y_m)\\
			 \hspace{.8cm} = \Theta(y_m^*  \cdots y_1^* (\mathring{a^*a})y_1\cdots y_m) &+ \varphi_{v_0}(a^*a)\Theta(y_m^*\cdots  y_1^*y_1\cdots y_m).
			\end{align*}
			Notice that $\nci (\mathring{a^*a})y_1\cdots y_J\nci_{v_J} > \nci y_1\cdots y_J\nci_{v_J}$.  Since we chose the largest possible $J$, \[(\mathring{a^*a})y_1\cdots y_{J-1} (y_J)(y_{J+1} \cdots y_m)\] is in standard form with respect to $v_J$.  So applying Lemma \ref{X1crossterms} twice, we get that 
			\begin{align*}
			& \Theta(y_m^* \cdots y_1^*a^*ay_1 \cdots  y_m)\\
			& = \Theta(y_m^*\cdots y_J^*)\Theta(y_{J-1}^*\cdots y_1^*(\mathring{a^*a})y_1\cdots y_{J-1})\Theta(y_J \cdots y_m)\\
			&+ \varphi_{v_0}(a^*a) \Theta(y_m^*\cdots y_J^* y_{J-1}^*\cdots y_1^*y_1\cdots y_{J-1} y_J \cdots y_m).
			\end{align*}
			By the same inductive argument as in the commuting case, the generalized Schwarz's Inequality for the strictly smaller complete set $\left\{y\right\}^\peq$ gives
			\begin{align}
			\hspace{.6cm}&\Theta(y_m^*\cdots y_J^*)\Theta(y_{J-1}^*\cdots y_1^*(\mathring{a^*a})y_1\cdots y_{J-1})\Theta(y_J \cdots y_m)\notag\\
			\hspace{.6cm} &+ \varphi_{v_0}(a^*a) \Theta(y_m^*\cdots y_J^* y_{J-1}^*\cdots y_1^*y_1\cdots y_{J-1} y_J \cdots y_m)\notag\\
			\hspace{.6cm}&\geq \Theta(y_m^*\cdots y_J^*)\Theta(y_{J-1}^*\cdots y_1^*(\mathring{a^*a})y_1\cdots y_{J-1})\Theta(y_J \cdots y_m)\notag\\
			\hspace{.6cm}& + \varphi_{v_0}(a^*a) \Theta(y_m^*\cdots y_J^*)\Theta(y_{J-1}^*\cdots y_1^*y_1\cdots y_{J-1})\Theta(y_J \cdots y_m)\notag\\
			\hspace{.6cm}&= \Theta(y_m^* \cdots y_J^*)\Theta(y_{J-1}^*\cdots y_1^* a^*a y_1 \cdots y_{J-1})\Theta(y_J \cdots y_m)\notag\\
			\hspace{.6cm}&\geq \Theta(y_m^* \cdots y_J^*)\Theta(y_{J-1}^*\cdots y_1^*)\Theta_{v_0}(a^*a)\Theta(y_1\cdots y_{J-1})\Theta(y_J\cdots y_M) \label{3}\\
			\hspace{.6cm}& = \Theta(y_m^*\cdots y_1^*)\Theta_{v_0}(a^*a)\Theta(y_1\cdots y_m)\notag
			\end{align}
			where \eqref{3} follows from the inductive hypothesis on $|y|$.\qedhere.
		\end{itemize}
\end{itemize}
\end{proof}

We use our generalized Schwarz's Inequality to prove the following lemma.

\begin{lem}\label{Y1square}
Let $\left\{ x_i\right\}_{i=1}^N \in (\mc{W}_\text{red}\cup \left\{1\right\})^N$ be a finite sequence such that for every $1 \leq i \leq N,$ we have  $v_0 \in \mbf{v}_{x_i}$.  For each $1 \leq i \leq N$, let $x_i = y_ic_ia_ib_i$ be in standard form with respect to $v_0$ ($a_i \in \mathring{\mc{A}}_{v_0}$).  Assume the following.
\begin{enumerate}
	\item For every $1 \leq i,j \leq N, \mbf{v}_{y_i} = \mbf{v}_{y_j}$;
	\item\label{sc} For every complete set $X \subsetneq (\left\{ x_i\right\}_{i=1}^N)^\peq$ and any function $\xi: X \rightarrow \mc{H}$, \eqref{goal} holds.
\end{enumerate}
Then \[\left[\Theta(x_i^*x_j)\right]_{ij} \geq \left[\Theta(b_i^*a_i^*)\Theta(c_i^*y_i^*y_jc_j)\Theta(a_jb_j)\right]_{ij}.\]
\end{lem}

\begin{proof}\hspace*{\fill}

%
	\begin{itemize}
		\item First suppose $\nci (\left\{x_i\right\}_{i=1}^N)^\peq\nci_{v_0} = 0$. Then for every $1 \leq i \leq N, y_i = 1$.  So $x_i = c_ia_ib_i$.  Standard form implies that for each $1 \leq i \leq N, c_i$ commutes with $a_i$.  Thus,
		\begin{align}
		&\Theta(b_i^*a_i^*c_i^*c_ja_jb_j) \notag\\
		&=\Theta(b_i^*a_i^*a_jc_i^*c_jb_j)&\notag\\
		&= \Theta(b_i^*(\mathring{a_i^*a_j})c_i^*c_jb_j) + \varphi_{v_0}(a_i^*a_j) \Theta(b_i^*c_i^*c_jb_j)\notag\\
		&= \Theta(b_i^*(\mathring{a_i^*a_j}))\Theta(c_i^*c_jb_j) + \varphi_{v_0}(a_i^*a_j)\Theta(b_i^*c_i^*c_jb_j) \label{4}\\
		&=\Theta(b_i^*)\Theta(\mathring{a_i^*a_j})\Theta(c_i^*c_jb_j) + \varphi_{v_0}(a_i^*a_j)\Theta(b_i^*c_i^*c_jb_j)\notag\\
		&=\Theta(b_i^*)\Theta((\mathring{a_i^*a_j})c_i^*c_jb_j) + \varphi_{v_0}(a_i^*a_j)\Theta(b_i^*c_i^*c_jb_j)\notag\\
		&=\Theta(b_i^*)\Theta(c_i^*c_j(\mathring{a_i^*a_j})b_j) + \varphi_{v_0}(a_i^*a_j)\Theta(b_i^*c_i^*c_jb_j)\notag\\
		&=\Theta(b_i^*)\Theta(c_i^*c_j)\Theta((\mathring{a_i^*a_j})b_j) + \varphi_{v_0}(a_i^*a_j)\Theta(b_i^*c_i^*c_jb_j)\label{5}\\
		&=\Theta(b_i^*)\Theta(c_i^*c_j)\Theta(\mathring{a_i^*a_j})\Theta(b_j) + \varphi_{v_0}(a_i^*a_j)\Theta(b_i^*c_i^*c_jb_j)\notag
		\end{align}
		where \eqref{4} and \eqref{5} follow from Lemma \ref{X1crossterms}.  Now, $\left\{c_ib_i\right\}_{i=1}^N$ is a sequence of elements from a complete set $X$ strictly contained in $(\left\{x_i\right\}_{i=1}^N)^\peq$. So by assumption \eqref{sc}, we have that \eqref{goal} holds for $X$, and so by Proposition \ref{schwarz}, we have \[\left[\Theta(b_i^*c_i^*c_jb_j)\right]_{ij} \geq \left[\Theta(b_i^*)\Theta(c_i^*c_j)\Theta(b_j)\right]_{ij}.\]  And since $[\varphi_{v_0}(a_i^*a_j)]_{ij}$ is positive and the $\varphi_{v_0}(a_i^*a_j)$'s are central, we have  by Lemma IV.4.24 in \cite{takesaki} that \[\left[\varphi_{v_0}(a_i^*a_j)\Theta(b_i^*c_i^*c_jb_j)\right]_{ij} \geq \left[\varphi_{v_0}(a_i^*a_j)\Theta(b_i^*)\Theta(c_i^*c_j)\Theta(b_j)\right]_{ij}.\]   Also, we have that $[\Theta(c_i^*c_j)]_{ij} \geq 0$ by \eqref{sc} and Proposition \ref{schwarz}, and again by the classical version of Schwarz's Inequality,
		\begin{align*}
		\left[\Theta(a_i^*a_j)\right]_{ij} & = \left[\theta_{v_0}(a_i^*a_j)\right]_{ij}\\
		&\geq \left[\theta_{v_0}(a_i^*)\theta_{v_0}(a_j)\right]_{ij}\\
		&= \left[\Theta(a_i^*)\Theta(a_j)\right]_{ij}.
		\end{align*}
		So since the $\Theta(c_i^*c_j)$'s commute with the $\Theta(a_i^*a_j)$'s and $\Theta(a_i^*)\Theta(a_j)$'s, then again by \cite{takesaki},
		\[\left[\Theta(c_i^*c_j)\Theta(a_i^*a_j)\right]_{ij} \geq \left[\Theta(c_i^*c_j)\Theta(a_i^*)\Theta(a_j)\right]_{ij}.\]
		Thus we have
		\begin{align*}
		&\left[\Theta(x_i^*x_j)\right]_{ij} \\
		&=\left[ \Theta(b_i^*)\Theta(c_i^*c_j)\Theta(\mathring{a_i^*a_j})\Theta(b_j)\right]_{ij}  + \left[ \varphi_{v_0}(a_i^*a_j)\Theta(b_i^*c_i^*c_jb_j)\right]_{ij}\\
		&\geq \left[ \Theta(b_i^*)\Theta(c_i^*c_j)\Theta(\mathring{a_i^*a_j})\Theta(b_j)\right]_{ij}  + \left[ \varphi_{v_0}(a_i^*a_j)\Theta(b_i^*)\Theta(c_i^*c_j)\Theta(b_j)\right]_{ij}\\
		& = \left[ \Theta(b_i^*)\Big(\Theta(c_i^*c_j)\Theta(a_i^*a_j)\Big)\Theta(b_j)\right]_{ij}\\
		& \geq \left[ \Theta(b_i^*)\Big(\Theta(c_i^*c_j)\Theta(a_i^*)\Theta(a_j)\Big)\Theta(b_j)\right]_{ij}\\
		& = \left[ \Theta(b_i^*)\Theta(a_i^*)\Theta(c_i^*c_j)\Theta(a_j)\Theta(b_j)\right]_{ij} \\
		& = \left[ \Theta(b_i^*a_i^*)\Theta(c_i^*c_j)\Theta(a_jb_j) \right]_{ij}.
		\end{align*}
		
		\item Now suppose that $\nci (\left\{x_i\right\}_{i=1}^N)^\peq \nci_{v_0} >0$.  Say that $y_i = y_1(i) \cdots y_m(i)$.  Observe that 
		\begin{align}
		&\Theta(b_i^*a_i^*c_i^*y_i^*y_jc_ja_jb_j) \notag\\
		&=\Theta(b_i^*a_i^*c_i^*y_m(i)^*\cdots y_2(i)^*(\mathring{y_1(i)^*y_1(j)})y_2(j)\cdots y_m(j)c_ja_jb_j)\notag \\
		& + \varphi_{\mbf{v}_{y_1(1)}}(y_1(i)^*y_1(j))\Theta(b_i^*a_i^*c_i^*y_m(i)^*\cdots y_2(i)^*y_2(j)\cdots y_m(j) c_j a_j b_j)\notag\\
		&= \Theta(b_i^*a_i^*)\Theta(c_i^*y_m(i)^*\cdots y_2(i)^*(\mathring{y_1(i)^*y_1(j)})y_2(j)\cdots y_m(j)c_j)\Theta(a_jb_j)\label{6} \\
		& + \varphi_{\mbf{v}_{y_1(1)}}(y_1(i)^*y_1(j))\Theta(b_i^*a_i^*c_i^*y_m(i)^*\cdots y_2(i)^*y_2(j)\cdots y_m(j) c_j a_j b_j)\notag
		\end{align}
		where \eqref{6} follows from Lemma \ref{Y1crossterms}.  Since $(\left\{ y_2(i)\cdots y_m(i)c_ia_ib_i\right\}_{i=1}^N)^\peq$ is a strictly smaller complete set, then assumption \eqref{sc} combined with Proposition \ref{schwarz} and \cite{takesaki} gives that 
		\begin{small}
		\begin{align*}
		\hspace{1cm}&\left[\varphi_{\mbf{v}_{y_1(1)}}(y_1(i)^*y_1(j))\Theta(b_i^*a_i^*c_i^*y_m(i)^*\cdots y_2(i)^*y_2(j)\cdots y_m(j) c_j a_j b_j)\right]_{ij}\\
		\hspace{1cm}&\geq \left[\varphi_{\mbf{v}_{y_1(1)}}(y_1(i)^*y_1(j))\Theta(b_i^*a_i^*)\Theta(c_i^*y_m(i)^*\cdots y_2(i)^*y_2(j)\cdots y_m(j) c_j)\Theta(a_j b_j)\right]_{ij}.
		\end{align*}
		\end{small}
		The desired inequality follows.\qedhere
	\end{itemize}
\end{proof}

\subsection{Proof of the Main Theorem} 
We are now adequately prepared to prove Theorem \ref{mainthm}.

\begin{proof}[Proof of Theorem \ref{mainthm}]
It will suffice to show that $\Theta$ is completely positive on the linear span of $\mc{W}_\text{red} \cup \left\{1\right\}$.  Indeed, Proposition 2.1 of \cite{paulsen} would then give that $\Theta$ is bounded and thus extends by continuity to a completely positive map on $\bigstar_\Gamma \mc{A}_v$.  

As discussed above, this problem reduces to showing that given a complete set $X \subseteq \mc{W}_\text{red}\cup \left\{1\right\}$ and any function $\xi: X \rightarrow \mc{H}$ the inequality \eqref{goal} holds.  We proceed by induction on $|X|$.

\begin{itemize}
	\item $|X| =1$: Trivial.
	
	\item $|X| \geq 2$: Let $(v_0) \in \mbf{v}_X$. Put \[X_1:=\left\{ x \in X \Big| \nci \left\{x\right\}^\peq \nci_{v_0} = \nci X \nci_{v_0}\right\},\] and let $x_0 \in X_1$ be an element of longest length in $X_1$.  Say that $x_0 = y_0c_0a_0b_0$ is in standard form with respect to $v_0$ (and so $a_0 \in \mathring{\mc{A}}_{v_0}$).  Define \[Y_1 := \left\{x \in X_1 \big| \text{ in standard form } x = ycab \,(a \in \mathring{\mc{A}}_{v_0}), \mbf{v}_y = \mbf{v}_{y_0}\right\}.\]  Note the following decomposition.
	\begin{align*}
	&\sum_{x,y \in X} \langle \Theta(x^*y)\xi(y) | \xi(x)\rangle \\
	& =\sum_{w,z \in X\setminus Y_1} \langle \Theta(w^*z)\xi(z) | \xi(w)\rangle\\
	& + \sum_{x, x' \in Y_1} \langle \Theta(x^*x')\xi(x') | \xi(x)\rangle\\
	& + \sum_{x \in Y_1, z \in X \setminus Y_1} 2\mathfrak{Re} \langle \Theta(x^*z)\xi(z) | \xi(x)\rangle.
	\end{align*}
	Consider $X \setminus Y_1 \subset (X\setminus Y_1)^\peq$.  By our choice of $x_0$, we have that $x_0 \notin (X\setminus Y_1)^\peq$, so the inductive hypothesis on $|X|$ applies to the strictly smaller complete set $(X\setminus Y_1)^\peq$.  By the discussion in \S\S\ref{stine}, there is a Hilbert space $\mc{K}$ and operators $V_w \in B(\mc{H},\mc{K})$ for every $w \in X\setminus Y_1$ such that $V_w^*V_z = \Theta(w^*z)$ for every $w,z \in X \setminus Y_1$. 
	
	For $x, x' \in Y_1$, let $x = ycab$ and $x' = y'c'a'b'$ be their standard forms with respect to $v_0$.  By Lemmas \ref{X1crossterms} and \ref{Y1crossterms}, we have that
	\begin{align*}
	&\sum_{x \in Y_1, z \in X \setminus Y_1} 2\mathfrak{Re} \langle \Theta(x^*z)\xi(z) | \xi(x)\rangle\\
	&= \sum_{ycab \in Y_1,  z \in X \setminus Y_1} 2\mathfrak{Re} \langle \Theta(b^*a^*)\Theta(c^*y^*z)\xi(z) | \xi(ycab)\rangle\\
	&= \sum_{ycab \in Y_1, z \in X \setminus Y_1} 2\mathfrak{Re} \langle V_z\xi(z) | V_{yc}\Theta(ab)\xi(ycab)\rangle.
	\end{align*}
	By Lemma \ref{Y1square}, we have that
	\begin{align*}
	&\sum_{x, x' \in Y_1} \langle \Theta(x^*x')\xi(x') | \xi(x)\rangle\\
	& \geq \sum_{x=ycab, x'=y'c'a'b' \in Y_1} \langle \Theta(b^*a^*)\Theta(c^*y^*y'c')\Theta(a'b')\xi(y'c'a'b') | \xi(ycab)\rangle\\
	&= \sum_{ycab, y'c'a'b' \in Y_1} \langle V_{y'c'}\Theta(a'b')\xi(y'c'a'b') | V_{yc}\Theta(ab)\xi(ycab)\rangle\\
	&= \Big|\Big|\sum_{ycab \in Y_1} V_{yc}\Theta(ab)\xi(ycab)\Big|\Big|^2.
	\end{align*} 
	We also have 
	\begin{align*}
	\sum_{w,z \in X\setminus Y_1} \langle \Theta(w^*z)\xi(z) | \xi(w)\rangle &= \sum_{w,z \in X\setminus Y_1} \langle V_w^*V_z \xi(z) | \xi(w)\rangle\\
	& = \sum_{w,z \in X\setminus Y_1} \langle V_z \xi(z) | V_w\xi(w)\rangle\\
	&= \Big|\Big|\sum_{w \in X\setminus Y_1} V_w \xi(w)\Big|\Big|^2
	\end{align*}
	Thus we have
	\begin{align*}
	&\sum_{x,y \in X} \langle \Theta(x^*y)\xi(y) | \xi(x)\rangle \\
	& =\sum_{w,z \in X\setminus Y_1} \langle \Theta(w^*z)\xi(z) | \xi(w)\rangle + \sum_{x, x' \in Y_1} \langle \Theta(x^*x')\xi(x') | \xi(x)\rangle \\
	&+ \sum_{x \in Y_1, z \in X \setminus Y_1} 2\mathfrak{Re} \langle \Theta(x^*z)\xi(z) | \xi(x)\rangle\\
	&\geq \Big|\Big|\sum_{w \in X\setminus Y_1} V_w \xi(w)\Big|\Big|^2 + \Big|\Big|\sum_{x=ycab \in Y_1} V_{yc}\Theta(ab)\xi(ycab)\Big|\Big|^2 \\
	&+  \sum_{x=ycab \in Y_1, z \in X \setminus Y_1} 2\mathfrak{Re} \langle V_z\xi(z) | V_{yc}\Theta(ab)\xi(ycab)\rangle\\
	&= \Big|\Big|\sum_{w \in X \setminus Y_1} V_w \xi(w) + \sum_{x=ycab \in Y_1} V_{yc}\Theta(ab) \xi(ycab)\Big|\Big|^2\\
	&\geq 0. \qedhere
	\end{align*}
\end{itemize}
\end{proof}


\subsection{Tensor Product Example}\label{tpe}  Due to the technical nature of the above proof, it is illustrative to write out the case where $\Gamma$ is a complete graph.  This gives a new combinatorial proof of the fact that the tensor product of ucp maps on the maximal tensor product of unital $C^*$-algebras is ucp.

Let $\mc{A}_v, \varphi_v, \theta_v, \mc{B}\subset B(\mc{H})$ be as in the statement of Theorem \ref{mainthm}, and suppose that $\Gamma$ is a complete graph. Let $\Theta:= \bigstar_\Gamma \theta_v$. We wish to show that for any complete set $X \subset \mc{W}_\text{red}\cup\left\{1\right\}$ and any function $\xi: X \rightarrow \mc{H}$ we have the following inequality. \[\sum_{x,y \in X} \langle \Theta(x^*y)\xi(y) | \xi(x)\rangle \geq 0\] Let $v_0 \in V$ be such that $(v_0) \in \mbf{v}_X$. We proceed by induction on $|X|$.  The base case is again trivial. Following the definitions in the proof of Theorem \ref{mainthm}, we have \[X_1 = Y_1 = \left\{ x \in X | v_0 \in \mbf{v}_x\right\};\] furthermore, for any $x \in Y_1, \mbf{v}_x = (\cdots, v_0)$ because $\Gamma$ is complete. So for any $x \in Y_1,$ we can write $x$ in standard form with respect to $v_0$ as follows.
\begin{align}
 x = ca \text{ where } &a \in \mathring{\mc{A}}_{v_0} \text{ and } v_0 \notin \mbf{v}_{c} \label{deco}
 \end{align} Again, consider the decomposition given by 
\begin{align*}
	&\sum_{x,y \in X} \langle \Theta(x^*y)\xi(y) | \xi(x)\rangle\\
	& =\sum_{w,z \in X\setminus Y_1} \langle \Theta(w^*z)\xi(z) | \xi(w)\rangle \\
	&+ \sum_{x, x' \in Y_1} \langle \Theta(x^*x')\xi(x') | \xi(x)\rangle \\
	&+ \sum_{x \in Y_1, z \in X \setminus Y_1} 2\mathfrak{Re} \langle \Theta(x^*z)\xi(z) | \xi(x)\rangle.
	\end{align*}
As before, we have 
\begin{align*}
\sum_{w,z \in X\setminus Y_1} \langle \Theta(w^*z)\xi(z) | \xi(w)\rangle &= \sum_{wz \in X \setminus Y_1} \langle V_w^*V_z \xi(z) | \xi(w) \rangle\\
&= \Big|\Big| \sum_{w \in X \setminus Y_1} V_w \xi(w)\Big|\Big|^2.
\end{align*}
By \eqref{deco}, it is clear that 
\begin{align*}
\sum_{x \in Y_1, z \in X \setminus Y_1} 2\mathfrak{Re} \langle \Theta(x^*z)\xi(z) | \xi(x)\rangle & = \sum_{x=ca \in Y_1, z \in X \setminus Y_1} 2\mathfrak{Re} \langle \Theta(a^*c^*z)\xi(z) | \xi(ca)\rangle\\
&= \sum_{ca \in Y_1, z \in X \setminus Y_1} 2\mathfrak{Re} \langle V_z\xi(z) | V_c \Theta(a) \xi(a)\rangle.
\end{align*}
Lastly we have
\begin{align}
\sum_{x, x' \in Y_1} \langle \Theta(x^*x')\xi(x') | \xi(x)\rangle & = \sum_{x=ca, x'=c'a' \in Y_1} \langle \Theta(a^*c^*c'a') \xi(c'a') | \xi(ca)\rangle \notag\\
& = \sum_{ca,c'a' \in Y_1} \langle \Theta(a^*a')\Theta(c^*c') \xi(c'a') | \xi(ca)\rangle \label{11}\\
& \geq \sum_{ca,c'a' \in Y_1} \langle \Theta(a^*)\Theta(a')\Theta(c^*c') \xi(c'a') | \xi(ca)\rangle \label{12}\\
&= \Big|\Big|\sum_{ca,c'a' \in Y_1} V_c\Theta(a) \xi(ca)\Big|\Big|^2 \notag
\end{align}
where \eqref{11} follows from the fact that $\Gamma$ is complete, and \eqref{12} follows from the classical Schwarz Inequality applied to the ucp map $\theta_{v_0}$ combined with Lemma IV.4.24 in \cite{takesaki}.  Combining these observations yields
\begin{align*}
&\sum_{x,y \in X} \langle \Theta(x^*y)\xi(y) | \xi(x)\rangle\\
 &\geq  \Big|\Big| \sum_{w \in X \setminus Y_1} V_w \xi(w)\Big|\Big|^2 + \Big|\Big|\sum_{ca,c'a' \in Y_1} V_c\Theta(a) \xi(ca)\Big|\Big|^2 \\
 &+  \sum_{ca \in Y_1, z \in X \setminus Y_1} 2\mathfrak{Re} \langle V_z\xi(z) | V_c \Theta(a) \xi(ca)\rangle\\ 
&= \Big|\Big|\sum_{w \in X \setminus Y_1} V_w \xi(w) + \sum_{ca,c'a' \in Y_1} V_c\Theta(a) \xi(ca)\Big|\Big|^2\\
&\geq 0.
\end{align*}

\section{Consequences}\label{cons}

\subsection{Reduced version}

We record the graph product version of Proposition 2.1 in \cite{choda}.  As in the amalgamated free product case, this result follows directly from Theorem \ref{mainthm}.  It should be noted that although the reduced version follows directly from Boca's result in the amlagamated free product setting, Choda's approach explicitly constructs a dilation on a Hilbert space containing the free product Hilbert space.  We present the graph product version as a direct corollary to Theorem \ref{mainthm}, but it is not unreasonable to expect that one can give a graph product adaptation of Choda's proof.

\begin{cor}\label{choda}
Let $\Gamma = (V, E)$ be a graph, and for each $v\in V$ let $\mc{A}_v$ and $\mc{B}_v$ be unital $C^*$-algebras with states $\varphi_v \in S(\mc{A}_v)$ and $\psi_v \in S(\mc{B}_v)$.  For each $v \in V$ let $\theta_v: \mc{A}_v \rightarrow \mc{B}_v$ be a unital completely positive map with $\psi_v \circ \theta_v = \varphi_v$.  Then there exists a unital completely positive map $\bigstar_\Gamma \theta_v: \bigstar_\Gamma \mc{A}_v \rightarrow \bigstar_\Gamma (\mc{B}_v, \psi_v)$ such that 
\begin{enumerate}
\item\label{A} $\bigstar_\Gamma \psi_v \circ \bigstar_\Gamma \theta_v = \bigstar_\Gamma \varphi_v$;
\item $\bigstar_\Gamma \theta_v (a_1\cdots a_n) = \theta_{v_1}(a_1)\cdots \theta_{v_n}(a_n)$ for $a_j \in \mathring{\mc{A}}_{v_j}, (v_1,\dots,v_n) \in \mc{W}_\text{red}$.
\end{enumerate}
\end{cor} 

\begin{proof}
Take $\bigstar_\Gamma \theta_v$ to be the graph product ucp map as in in \ref{mainthm} defined with respect to the states $\varphi_v$. Part \eqref{A} follows from Lemma \ref{stres}.
\end{proof}


\subsection{Graph products of positive-definite functions}

We show here that the graph product of positive-definite functions is positive-definite itself.  This is a graph product version of Theorem 7.1 in \cite{boz}.

\begin{dfn}
Let $G$ be a group and $\mc{H}$ be a Hilbert space.  A function $f: G \rightarrow B(\mc{H})$ is \emph{positive-definite} if for every finite subset $\left\{g_1,\dots, g_n\right\} \subset G$, the $n\times n$ matrix \[\left[ f(g_i^{-1}g_j)\right]_{ij}\] is positive.
\end{dfn}

\begin{dfn}\label{gpposdef}
Let $\mc{H}$ be a Hilbert space, and for each $v \in V$, let $G_v$ be a group and $f_v: G_v \rightarrow B(\mc{H})$ be positive-definite with $f_v(e) = 1$.  If $(v,v') \in E \Rightarrow f_v(G_v)$ and $f_{v'}(G_{v'})$ commute, then we define the graph product of the $f_v$'s, $\bigstar_\Gamma f_v: \bigstar_\Gamma G_v \rightarrow B(\mc{H})$, as follows. 
\begin{enumerate}
\item $\bigstar_\Gamma f_v (e) = 1$;
\item if for $1 \leq k \leq n,\; g_k \in G_{v_k}\setminus\left\{1\right\}$ and $(v_1,\dots,v_n) \in \mc{W}_\text{red}$, then \[\bigstar_\Gamma f_v (g_1\cdots g_n) := f_{v_1}(g_1)\cdots f_{v_n}(g_n).\]
 \end{enumerate} 
\end{dfn}

It is well-known that there is a 1-1 correspondence between positive-definite functions $f:G \rightarrow B(\mc{H}), f(e) =1$   and ucp maps $\theta: C^*(G) \rightarrow B(\mc{H})$ in the following sense.  If $u_g \in C^*(G)$ denotes the unitary corresponding to the group element $g \in G$, then
\begin{align*}
f &\rightarrow \theta_f(u_g) := f(g)\\
f_\theta(g):= \theta(u_g) &\leftarrow \theta.
\end{align*}

\begin{thm}\label{posdef}
Let $G_v, f_v$ and $\mc{H}$ be as in Definition \ref{gpposdef}.  Then $\bigstar_\Gamma f_v$ is positive-definite.
\end{thm}

\begin{proof}
Let $\bigstar_\Gamma \theta_{f_v}$ be the graph product of the ucp maps on $C^*(G_v)$ corresponding to $f_v$ defined with respect to states given by the canonical traces (from the left-regular representation) on $C^*(G_v)$.  By Theorem \ref{mainthm}, $\bigstar_\Gamma \theta_{f_v}$ is ucp. Then it is easy to check that $f_{\bigstar_\Gamma \theta_{f_v}} = \bigstar_\Gamma f_v$.
\end{proof}

\subsection{Unitary dilation}

We conclude the paper with some results on unitary dilation in the graph product context.  Consider the following version of the Sz.-Nagy-Foia\lfhook{s} dilation theorem.

\begin{thm}\label{dil}
Let $\Gamma = (V,E)$ be a graph.  Let $\mc{H}$ be a Hilbert space and $\left\{T_v \right\}_{v \in V} \subset B(\mc{H})$ be contractions such that if $(v,v') \in E$ then $T_v$ and $T_{v'}$ doubly commute ($[T_v,T_{v'}] = [T_v^*,T_{v'}] = 0$).  Then there exist a Hilbert space $\mc{K}$ containing $\mc{H}$ and unitaries $U_v \in B(\mc{K})$ for each $v \in V$ such that for any polynomial $p \in \mb{C}\langle X_v \rangle_{v\in V}$ in $|V|$ non-commuting indeterminates we have \[ p(\left\{A_v\right\}_{v\in V}) = P_\mc{H} p(\left\{U_v\right\}_{v\in V}) |_\mc{H}.\]
\end{thm}

\begin{proof}
 By Stinespring, we will be done if we obtain a ucp map $\Theta: \bigstar_\Gamma C^*(\mb{Z})  \rightarrow B(\mc{H})$ such that $\Theta(p((x_v)) = p((T_v))$.  Indeed, let $U_v$ be the image of $x_v$ under the resulting Stinespring representation.  
 
 Define the ucp map $\theta_{v}$ on the $v^\text{th}$ copy of $C^*(\mb{Z})$ as follows. \[\theta_{v}(x_v^m) = \left\{ \begin{array}{lcr} T_v^m & \text{if} & m \geq 0\\ (T_v^*)^{-m} & \text{if} & m <0\end{array}\right.\]  (This map is ucp by Sz.-Nagy's unitary dilation theorem).  Then the map $\Theta = \bigstar_\Gamma \theta_{v} : \bigstar_\Gamma C^*(\mb{Z}) = C^*(\bigstar_\Gamma \mb{Z}) \rightarrow B(\mc{H})$ defined with respect to the canoncial trace on $C^*(\mb{Z})$ does the job.
\end{proof}

\begin{rmk}
It should be emphasized that the doubly commuting assumption is important for the above theorem.  In particular, Op\v{e}la showed in Theorem 2.3 of \cite{opela} that if $\Gamma = (V, E)$ is a graph with $n \in \mb{N}$ vertices containing a cycle (a closed path of edges) then there are contractions $T_1, \dots, T_n$ such that if $(v_i,v_j) \in E$ then $[T_i,T_j] =0$ (not doubly commuting) with no simultaneous unitary dilation.  On the other hand, if $\Gamma$ has no cycles, then plain (single) commutation relations according to $\Gamma$ can be dilated.
\end{rmk}

The following corollary is a graph product version of Theorem 8.1 of \cite{boz} and follows immediately from Theorem \ref{dil}.  First a definition is in order.

\begin{dfn}
Given a graph $\Gamma = (V,E)$, let $\bigstar_\Gamma \mb{Z}$ denote the graph product group $\bigstar_\Gamma G_v$ where $G_v = \mb{Z}$ for every $v \in V$.  This is the graph product analog of $\mb{F}_n$.
\end{dfn}

\begin{cor}\label{vnineq}
Let $\Gamma = (V,E)$ be a graph.  Let $\mc{H}$ be a Hilbert space and $\left\{T_v \right\}_{v \in V} \subset B(\mc{H})$ be contractions such that if $(v,v') \in E$ then $T_v$ and $T_{v'}$ doubly commute ($[T_v,T_{v'}] = [T_v^*,T_{v'}] = 0$).  Let $p \in \mb{C}\langle X_v\rangle_{v\in V}$ be a polynomial in $|V|$ non-commuting indeterminates.  Then \[||p(\left\{T_v\right\}_{v \in V})|| \leq ||p(\left\{x_v\right\}_{v\in V})||_{C^*(\bigstar_\Gamma \mb{Z})}\] where for each $v \in V$, $x_v$ denotes the unitary corresponding to the canonical generator of the $v^\text{th}$ copy of $\mb{Z}$.
\end{cor}

\begin{rmk}
Note that by the universality of $C^*(\mb{F}_{|V|})$ we have \[||p||_{C^*(\bigstar_\Gamma \mb{Z})} \leq ||p||_{C^*(\mbf{F}_{|V|})}.\]  
\end{rmk}

Lastly, we have a version of Theorem \ref{dil} viewed through the lens of non-commutative probability.  The statement and proof are simple adaptations of the free versions presented in \cite{sacr}.

\begin{thm}\label{gpdilation}
 Given a graph $\Gamma = (V, E)$ and $\Gamma$ independent contractions $\left\{T_v\right\}_{v \in V}$ in the noncommutative probability space $(B(\mc{H}), \varphi)$, there exist a Hilbert space $\mc{K}$ containing $\mc{H}$ and unitaries $\left\{U_v\right\}_{v\in V} \subset B(\mc{K})$ that are $\Gamma$ independent with respect to $\varphi \circ \text{Ad}(P_\mc{H})$ such that for any polynomial $p \in \mb{C}\langle X_v \rangle_{v\in V}$ in $|V|$ non-commuting indeterminates we have \[ p(\left\{T_v\right\}_{v\in V}) = P_\mc{H} p(\left\{U_v\right\}_{v\in V}) |_\mc{H}.\]  Furthermore, this dilation is unique up to unitary equivalence if $\mc{K}$ is minimal.
\end{thm}

\begin{proof}
We use the same dilation as in Theorem \ref{dil}, letting $\pi: \bigstar_\Gamma C^*( \mb{Z}) \rightarrow B(\mc{K})$ denote the corresponding Stinespring representation; and for every $v \in V$ let $U_v = \pi(x_v)$ where $x_v$ is unitary corresponding to the canonical generator of the $v^\text{th}$ copy of $\mb{Z}$.  It remains to show the $\Gamma$ independence of $\left\{ U_v \right\}_{v\in V} \subset B(\mc{K})$ and uniqueness in the case that $\mc{K}$ is minimal.

To show that the random variables in $\left\{U_v\right\}_{v \in V}$ are $\Gamma$ independent with respect to $\varphi \circ \text{Ad}(P_\mc{H})$, let $a = a_1\cdots a_m$ where $a_j \in \mathring{C^*(U_{v_j})}$ for $1 \leq j \leq m$ be reduced with respect to $\varphi \circ \text{Ad}(P_\mc{H})$. For $1 \leq j \leq m$, let $b_j$ be an element of the $v_j^\text{th}$ copy of $C^*(\mb{Z})$ such that $\pi(b_j) = a_j$.  It follows that \[\bigstar_\Gamma \theta_v(b_1\cdots b_m) = \theta_{v_1}(b_1)\cdots \theta_{v_m}(b_m)\] is reduced with respect to $\varphi$.  Then by the $\Gamma$ independence of $\left\{T_v\right\}_{v\in V}$, we have
\begin{align*}
\varphi(P_\mc{H} a_1\cdots a_m |_\mc{H}) &= \varphi(P_\mc{H} \pi(b_1\cdots b_m)|_\mc{H}) \\
&= \varphi(\bigstar_\Gamma \theta_v (b_1\cdots b_m))\\
&= \varphi(\theta_{v_1}(b_1)\cdots \theta_{v_m}(b_m))\\
&=0.
\end{align*}

The minimality argument follows from the same argument presented in the proof of Theorem 3.2 in \cite{sacr} using Lemma \ref{stres} in place of Lemma 5.13 from \cite{spenic}.
\end{proof}

\begin{rmk}\hspace*{\fill}
\begin{enumerate}

\item If $\Gamma$ is complete then, as shown in \cite{nagy-foias, sacr}, we can take $p$ to be a $*$-polynomial.

\item By Theorem 1 in \cite{mlot}, we have that $\varphi \circ \text{Ad}(P_\mc{H})$ is tracial on $C^*(\left\{U_v\right\}_{v \in V})$.
\end{enumerate}
\end{rmk}

\subsection*{Acknowledgments}
Gratitude is due to Ben Hayes for initiating the author's interest in this subject and to David Sherman for valuable conversations about this project.  Also, the author would like to thank Andrew Sale for providing helpful information on the relevant group theoretic literature. Because of a gracious invitation, a portion of this article was completed during a June 2017 visit to the Centre de Recerca Matem\`{a}tica in Barcelona, Spain.

\bibliographystyle{plain}
\bibliography{gpucpbib}{}

\end{document}